\pdfoutput=1
\documentclass[oneside,reqno,12pt]{amsart}

\usepackage{MyPackages_arXiv}
\usepackage{MyMathSymbols}
\usepackage{MyMathFonts}
\usepackage{MyNewTheorems}
\usepackage{biolinum} 

\numberwithin{equation}{section}

\begin{document}


\author[Hyunbin Kim]{Hyunbin Kim}
\address{Department of Mathematics \ Yonsei University \ 50 Yonsei-Ro \ Seodaemun-Gu \ Seoul 03722 \ Korea} 
\email{hyunbinkim@yonsei.ac.kr}

\title{Closed-String Mirror Symmetry for Log Calabi-Yau Surfaces}

\begin{abstract}
This paper establishes closed-string mirror symmetry for all log Calabi-Yau surfaces with generic parameters, where the exceptional divisors are sufficiently small. We demonstrate that blowing down a $(-1)$-divisor removes a single geometric critical point, ensuring that the resulting potential remains a Morse function. Additionally, we show that the critical values are distinct, which implies that the quantum cohomology $QH^{\ast}(X)$ is semi-simple.
\end{abstract}

\maketitle
\setcounter{tocdepth}{1}
\tableofcontents



\section{Introduction}

Given a special Lagrangian fibration on a Calabi-Yau manifold, the Strominger-Yau-Zaslow conjecture \cite{SYZ96} (SYZ conjecture for short) suggests that the mirror pair should be constructed by taking fiber-wise dual. The SYZ approach has been since used extensively to construct mirrors beyond the realm of Calabi-Yau manifolds, especially for \textit{log Calabi-Yau surfaces}, notably in \cite{auroux07}, \cite{BCHL}, \cite{GHK}, \cite{HK}, \cite{LaiZhou}, etc.

A log Calabi-Yau surface (also known as a Looijenga pair) is a rational projective surface $X$, together with a reduced rational curve $D \in \lvert -K_{X} \rvert$ with at least one singular point. To construct the mirror of $(X, D)$, one typically starts by taking the fiber-wise dual of the special Lagrangian torus fibration (usually given as the moment map) on the complement of the anticanonical divisor $ Y = X \setminus D$. The divisor $D$ is then accounted for by equipping $\check{Y}$ with a Laurent polynomial $W : \check{Y} \rightarrow \Lambda$ from the mirror to the Novikov field (see \ref{notation} for the definition), called the \textit{potential}. The potential records holomorphic disks with Maslov index $2$, provided that the torus fibers are weakly unobstructed in the sense of \cite{FOOO}. More recently, a Family Floer theoretic formulation of the SYZ conjecture has been introduced (\cite{Ab-famFl1}, \cite{Yuan}), where the mirror $\check{Y}$ is constructed as a rigid analytic variety; this perspective is adopted throughout this paper. Mirror symmetry then predicts that the singularity information of $W$, i.e. the Jacobian ring $\mathrm{Jac}(W)$ is equivalent to the symplectic information of $X$, encoded as the quantum cohomology $QH^{\ast}(X)$.

Studying blowups and blowdowns of log Calabi-Yau surfaces is a natural extension of this framework, especially since log Calabi-Yau surfaces can be systematically constructed through a series of blowups and blowdowns on a toric model \cite[Proposition 1.3]{GHK}. Blowing up a non-torus fixed point on $D$ results in a \textit{wall}, a codimension $1$ family of Lagrangians which bounds Maslov index $0$ disks. Such walls cause the potential to differ from one chamber to another, related by the \textit{wall-crossing formula}. In general, walls collide and generate more walls, and the wall structure is explicitly calculated using \textit{scattering diagrams}. In \cite{GHK}, a tropical method of constructing family mirrors and the canonical scattering diagram for log Calabi-Yau surfaces is presented, which was later recovered in the context of Lagrangian Floer theory in \cite{BCHL}.

For Fano and semi-Fano surfaces, one can directly work with the full potential, as it can be explicitly written down as shown in \cite{CO}, \cite{CLL}, and \cite{BCHL}. However, for non-Fano surfaces, the full potential may become an infinite series due to contributions from negative sphere bubbles, and explicit formulas are generally unknown (apart from some examples, e.g. \cite{auroux07}).  Despite this, only a few leading order terms are responsible for critical points that are \textit{geometric}, in the sense that they actually support Lagrangian torus fibers of  $\check{Y}$ (referred to as \textit{geometric critical points}). Provided that the Lagrangian fibers are weakly unobstructed, it is then feasible to work with non-Fano surfaces. Generically speaking, blowing up a point on $D$ introduces multiple terms (corresponding to broken disks and disks with sphere bubbles) to the potential, but only one of these terms gives rise to a geometric critical point, which is in fact non-degenerate \cite{HK23}. Thus, inductively, the potential $W$ of a log Calabi-Yau surface obtained through a series of blowups is generically a Morse function.\\

In this paper, we examine the remaining classes of log Calabi-Yau surfaces, namely those obtained by blowing down $(-1)$-divisors, and complete the picture of mirror symmetry for all log Calabi-Yau surfaces:
\begin{named}{Theorem I}[\cref{thm:main}]\label{thm:intromain}
Let $(X,D)$ be a log Calabi-Yau surface with generic parameters, i.e. $\omega$ satisfies \ref{assumption}. Then, 
\[
	\mathrm{Jac} ( W ) \cong QH^{\ast} ( X ).
\]
\end{named}

\noindent Blowing down a $(-1)$-divisor reduces the rank of the quantum cohomology by one, suggesting a corresponding decrease in the number of geometric critical points. However, unlike blowups, blowdowns cause an abrupt change in the leading order terms of the potential by removing multiple terms. This complicates the analysis of the potential's singularity information and initially suggests that more than one critical point is lost. Nonetheless, we find that some previously non-leading order terms become leading order terms, and ultimately, only a single monomial is lost. After obtaining the correct potential, an affine coordinate change is introduced on the base to accurately track changes in the domain of the potential function (i.e. the moment polytope) during the blowdown process.

Another crucial step in the proof of \ref{thm:intromain} is establishing the semi-simplicity of the quantum cohomology ring for specific symplectic forms:
\begin{named}{Theorem II}[\cref{prop:semisimple}]\label{thm:intromain2}
Let $(X, D)$ be a  log Calabi-Yau surface with generic coefficients, i.e. $\omega$ satisfies \ref{assumption}. Then the quantum cohomology $QH^{\ast}( X )$ is semi-simple.
\end{named}
\noindent The proof of \ref{thm:intromain2} relies on the crucial observation made by Auroux \cite{auroux07} and many others, that the critical value of the potential equals the eigenvalue of the quantum multiplication by $c_{1}(X)$.\\

The paper is organized as follow. In \cref{sec:2}, we review the basic construction of the SYZ mirror for log Calabi-Yau surfaces and briefly recall basic geometric setups and results from \cite{HK23}. In \cref{sec:3}, we prove weakly unobstructedness of Lagrangian fibers and demonstrate that blowdown results in a decrease in the number of geometric critical points by one. Finally, in \cref{sec:4}, we prove our two main theorems \ref{thm:intromain} and \ref{thm:intromain2}, establishing semi-simplicity of the quantum cohomology and closed string mirror symmetry for log Calabi-Yau surfaces.\\

\begin{notation}{Notation}\label{notation}
We fix some standard notations used throughout:
\begin{align*}
&\Lambda:=
\left\{ \;
\sum_{i=0}^{\infty} c_i T^{\lambda_i} 
\;\, \bigg| \;\,
c_i \in \mathbb{C},\, \lim_{i\to \infty} \lambda_i = \infty 
\;\right\},
\\
&\Lambda_0:= 
\left\{\;
\sum_{i=0}^{\infty} c_i T^{\lambda_i} \in \Lambda
\;\, \bigg| \;\,
\lambda_i \geq 0 
\;\right\},
\\
&\Lambda_+:= 
\left\{\;
\sum_{i=0}^{\infty} c_i T^{\lambda_i} \in \Lambda
\;\, \bigg| \;\,
\lambda_i > 0
\;\right\},
\\
&\Lambda_U:=
\mathbb{C}^\ast \oplus \Lambda_+ .
\end{align*}
The valuation map $\mathrm{val}  : \Lambda \to \mathbb{R}$ is defined by
\[
	\mathrm{val}:  \sum_{i=0}^{\infty} c_i T^{\lambda_i} \mapsto \min_i \left\{\;\lambda_i \;\, \Big| \;\, c_{i} \neq 0, \,\, i=0,1,2,\cdots\;\right\}.
\]
\end{notation}

\begin{center}
{\bf Acknowledgement}
\end{center}

I thank my advisor, Professor Hansol Hong, for his support and guidance. I also thank Professor Kaoru Ono, Yu-Shen Lin, Siu Cheong Lau, and Sam Bardwell-Evans for their valuable discussions, as well as the anonymous referees for their helpful comments. This research was supported by the National Research Foundation of Korea (NRF) grant funded by the Korean government (MSIT) (No. 2020R1C1C1A01008261 and No. 2020R1A5A1016126).




\section{Critical points of Mirror Landau-Ginzburg Potentials}\label{sec:2}

We start by reviewing Auroux's construction of the special Lagrangian fibration on an anticanonical divisor complement and its associated mirror LG model. The rest of the section revisits \cite{HK23}, which will be taken as the basic geometric setup for this paper.

\subsection{SYZ Mirror Construction and Lagrangian Floer Theory}

Consider a special Lagrangian fibration $\varphi : X \backslash D \to B$ on the complement $X \backslash D$ of an anticanonical divisor $D \subset X$. For a point $u \in B$, let $\{e_{1},\cdots, e_{n}\}$ be a basis for $H_{n-1} (L_u;\mathbb{Z}) =H^1 (L_u;\mathbb{Z})$, and let $\{f_{1},\cdots, f_{n}\}$ be its dual basis of $H_{1} (L_u;\mathbb{Z})$. The complex and symplectic affine coordinates on $B$ are respectively given by 
\[
	 u_i (u) := \int_{C_i} \mathrm{Im}\, \Omega,
	 \qquad 
	 x_i (u)=  \int_{A_i} \omega
	 \qquad i=1,\cdots, n,
\]
where $\Omega$ is the volume form determined by $D$ which has simple poles along $D$ and is holomorphic on $X \backslash D$. Here, $C_i$ is an $n$-dimensional chain swept out by $e_i$ along a path from the fixed reference point $u_0 \in B$ to $u$, and $A_i$ is a cylinder obtained analogously.

The mirror is the total space $\check{Y}^\mathbb{C}$ of the dual torus fibration $\check{\varphi}^{\mathbb{C}} : \check{Y}^\mathbb{C} \rightarrow B$, along with a Laurent series $W^\mathbb{C}$, called the \emph{potential}, determined by the count of holomorphic disks that intersect $D$. Explicitly, the mirror $\check{Y}^\mathbb{C}$ is 
\[
	\check{Y}^{\mathbb{C}} = 
	\left\{
	\left( L_u:=\varphi^{-1} (u),\nabla \right) 
	\;\; \Big| \;\; 
	u \in B, \;
	\nabla \in \Hom \left( H_{1}(L_u),  U(1) \right)
	\right\}.
\]
The potential $W^\mathbb{C} :\check{Y}^\mathbb{C} \to \mathbb{C}$ can be then written as
\begin{equation}\label{eqn:Wanticancc}
W^\mathbb{C} \left( L_u,\nabla \right) = \mathlarger{\sum}_{\substack{\mu(\beta)=2, \\ \beta \in \pi_{2} \left( X, L_{u} \right)}} N_\beta  e^{ - \int_\beta \omega} hol_{\partial \beta} \nabla,
\end{equation}
where $N_\beta$ is the number of holomorphic disks bounding $L_u$ in class $\beta$ (passing through a generic point of $L_u$). If $X$ is toric Fano, then $N_\beta$ counts disks intersecting $D$ exactly once, and \cref{eqn:Wanticancc} is a finite sum. \cref{eqn:Wanticancc} is also a finite sum if $X$ is a semi-Fano toric surfaces \cite{chan-lau}, with additional terms responsible for disks with sphere bubbling. In general, the non-Archimedean valuation ring $\Lambda$ is introduced, and \cref{eqn:Wanticancc} becomes
\begin{equation}\label{eqn:Wanticanll}
W \left( L_u,\nabla \right) =\mathlarger{\sum}_{\substack{\mu(\beta)=2, \\ \beta \in \pi_{2} \left( X, L_{u} \right)}} N_\beta  T^{  \int_\beta \omega} hol_{\partial \beta} \nabla,
\end{equation}
where $e^{-1}$ has been substituted by the formal variable $T$. Note that \cref{eqn:Wanticanll} always converges in the $T$-adic topology. Using local coordinates
\begin{equation*}\label{eqn:mirrorzi}
z_i=T^{x_i} hol_{f_i} \nabla,
\end{equation*}
the potential $W$ can be written as
\begin{equation}\label{eqn:Wzi}
W(z_{1},\cdots, z_{n})=\sum N_\beta T^{\int_{A_{\partial \beta}} \omega}  z^{\partial \beta},
\end{equation}
where $z^{\gamma} := z_{1}^{(\gamma,e_{1})} \cdots z_{n}^{(\gamma,e_{n})}$ for $\gamma \in H_{1} (L_u;\mathbb{Z})$, and $T^{\int_{A_{\partial{\beta}}} \omega}$ is the flux between $L_{u_0}$ and $L_u$ defined similarly to the symplectic affine coordinates.

Provided that the Lagrangian fiber $L_{u}$ is weakly unobstructed (see \cref{lemma:unobstructedness}), the mirror potential \cref{eqn:Wzi} can also be thought of as the local restriction $W_{u}$ of the Lagrangian Floer potential, defined implicitly by
\begin{equation}\label{eqn:Wfiber}
\sum_{k}m_{k}\left( b, \cdots, b \right) = W_{u}(b) \cdot [ L_{u} ], \qquad b = \sum y_{i}e_{i} \in H^{1} (L_{u} ; \Lambda_{+}).
\end{equation}
Allowing $W$ in \cref{eqn:Wanticanll} to be equipped with $\Lambda_U$-connection, for $(\underline{z}_{1}, \dots, \underline{z}_{n}) := (e^{y_{1}}, \dots, e^{y_{n}}) \in \Lambda_{U}$, \cref{eqn:Wanticanll} and \cref{eqn:Wzi} are related by 

\begin{equation}\label{eqn:relation}
	W_{u}(\underline{z}_{1}, \cdots, \underline{z}_{n}) 
	= W\left( L_{u}, \nabla^{(\underline{z}_{1}, \cdots, \underline{z}_{n})} \right),
\end{equation}
where $\nabla^{(\underline{z}_{1}, \cdots, \underline{z}_{n})} \in \Hom (H_{1}(L_u) , \Lambda_U)$ is a flat connection with holonomy $\underline{z}_{i}$ along the dual of $e_{i}$. The mirror $\check{Y}^{\mathbb{C}}$ is replaced by the rigid analytic variety $\check{Y}$, equipped with the dual torus fibration $\check{\varphi} :\check{Y} \to B$ with $(\Lambda_U)^n$-fibers over the same base. The fibration map $\check{\varphi}$ can be identified with the restriction $\mathrm{val} \big|_{B} : \check{\varphi}^{-1}(B) = \check{Y} \subset (\Lambda^\times)^n \to \mathbb{R}^n$ of the valuation map when written in terms of coordinates $z_i$. The readers are referred to \cite{Yuan} for a detailed construction of the rigid analytic variety.

\subsection{Blowups of Surfaces and Wall-Crossing}
Recall that \cite[Proposition 1.3]{GHK} states that there exists a toric blowup $(\widetilde{X}, \widetilde{D})$  of $(X,D)$ which has a toric model $(\widetilde{X}, \widetilde{D}) \rightarrow (X_{\Sigma}, D_{\Sigma})$. This implies that any log Calabi-Yau surface $(X,D)$ can be obtained by a sequence of toric/non-toric blowups $\pi$ on the toric surface $(X_{\Sigma}, D_{\Sigma})$, followed by another sequence of toric blowdowns $\widetilde{\pi}$ on $(\widetilde{X}, \widetilde{D})$. Namely, for any $(X, D)$, one can find a diagram
\begin{center}
\begin{tikzcd}\label{eqn:looijenga}
  & (\widetilde{X}, \widetilde{D}) \arrow[ld, "\widetilde{\pi}"'] \arrow[rd,"\pi"] &   \\
(X, D) &                         & \left( X_{\Sigma}, D_{\Sigma} \right)
\end{tikzcd}
\end{center}
where $\pi$ is the sequence of blowups, and $\widetilde{\pi}$ is the sequence of blowdowns\footnote{The theorems and arguments throughout this paper and those from \cite{HK23} can be applied back and forth for the``zig-zag'' process in the proof of \cite[Proposition 1.3]{GHK}.}.\\

\subsubsection{The Toric Model} Let $\left( X_{\Sigma}, \omega_{\Sigma} \right)$ be a symplectic toric manifold of (real) dimension $2n$, i.e. equipped with an effective Hamiltonian half-dimensional torus $T^{n}$-action. Recall that the corresponding moment polytope $\Delta$ is given by the inequalities
\[
	\Delta = 
	\left\{\,
	\mathbf{x} = (x_{1}, \dots, x_{n}) \in \mathbb{R}
	\;\,\Big|\;\,
	\langle \mathbf{x}, \, \nu_{i} \rangle \geq - \lambda_{i},
	\; i = 1, \dots, N
	\,\right\},
\]
where $\nu_{i} \in \mathbb{Z}^{n}$ are the primitive rays of the toric fan, and $N$ is the number of facets, i.e. $(n-1)$-dimensional faces of the moment polytope $\Delta$. 

We will restrict ourselves to the case $n=2$ from this point forward, as we are interested in log Calabi-Yau \textit{surfaces}. We will further assume that $D_{\Sigma,1}$ and $D_{\Sigma,2}$ are parallel to the coordinate axis, that is, $\nu_{1} = (1,0)$ and $\nu_{2}=(0,1)$, and that $\lambda_{1}=\lambda_{2}=0$, unless stated otherwise.

The preimage of the $i$-th facet under the moment map is a (real) codimension $2$ submanifold of $X_{\Sigma}$, which is precisely the boundary divisor of $X_{\Sigma}$. Let us denote by $[D_{\Sigma,i}] \in \mathrm{H}_{n-2}\left( X_{\Sigma};\mathbb{Z} \right)$ for the homology class of the boundary divisor corresponding to a generator $\nu_{i}$. Then the cohomology class of the symplectic form $\omega$ is given by the linear combination
\[
	[\omega_{\Sigma}] = \sum_{i=1}^{N} \lambda_{i} [D_{\Sigma,i}]^{\vee} = \sum_{i=3}^{N} \lambda_{i} [D_{\Sigma,i}]^{\vee},
\]
where $[D_{\Sigma,i}]^{\vee} \in \mathrm{H}^{2}\left( X_{\Sigma}; \mathbb{Z} \right)$ is the Poincar\'e dual of the homology class $[D_{\Sigma,i}]$.

Recall that a \emph{basic disk} $\beta_{\nu_{j}} \in \mathrm{H}_{2}(X_{\Sigma},L)$ is a Maslov index $2$ holomorphic disk intersecting the toric divisor $D_{\Sigma,i}$ exactly once, without any sphere bubbles. If $X_{\Sigma}$ is \textit{Fano}, \cite{CO} gives us the formula 
\[
	W_{\Sigma} = \sum_{j=1}^{N}T^{\lambda_{j}}z^{\nu_{j}}
\]
for the potential, which coincides with the Hori-Vafa potential $W^{HV}_{\Sigma}$. If $X_{\Sigma}$ is not Fano (assuming weakly unobstructedness of Lagrangian fibers), the potential function can be written as
\begin{equation*}\label{eqn:toricpotential}
W_{\Sigma}
= \underbrace{\sum_{j=1}^{N}T^{\lambda_{j}}z^{\nu_{j}}}_{ W^{HV}_{\Sigma} } 
+ \sum N_{\beta} T^{\delta(\beta)}z^{\partial \beta},
\end{equation*}
where the latter summand runs over rest of the Maslov index $2$ disks with sphere bubbles (in which case the disk component might have higher Maslov index). There are only a handful of cases where $W_{\Sigma}$ can be explicitly written when $X_{\Sigma}$ is non-Fano, some notable examples can be found in \cite{chan-lau} and \cite{auroux09}.

\subsubsection{Non-toric Blowups of $\left( X_{\Sigma},D_{\Sigma} \right)$}

Now suppose that $\widetilde{X}$ is obtained from $X_{\Sigma}$ after a sequence of non-toric blowups. Let us denote by $p_{j} \in X_{\Sigma}$ for the blowup centers, and $E_{j} \subset \widetilde{X}$ for their associated exceptional divisors. We do not distinguish the toric divisors $D_{\Sigma, j}$ and their proper transformation for notational simplicity. We equip $X$ with the symplectic form $\widetilde{\omega}$, in the class
\begin{equation}\begin{aligned}\label{eq:sympform}
[\widetilde{\omega}] 	&= \pi^{\ast}[\omega_{\Sigma}]  - \sum_{j=1}^{k} \epsilon_{j}[E_{j}]^{\vee}\\
		&= \sum_{i = 3}^{N} \lambda_{i}\pi^{\ast}[D_{\Sigma,i}]^{\vee} - \sum_{j=1}^{k} \epsilon_{j}[E_{j}]^{\vee},
\end{aligned}\end{equation}
where $\epsilon_{j}$'s are generic and $0 < \epsilon_{j}\ll 1$ (see \ref{assumption}).
Each non-toric blowup center yields a single nodal fiber in the torus fibration, which results in a \textit{wall} - a codimension $1$ family of Lagrangians which bounds a holomorphic disk of Maslov index $0$. The special Lagrangian torus fibration is constructed on the divisor complement, away from a small neighborhood of each branch cut, by gluing in local models introduced in \cite{auroux09} and \cite{BCHL}. If $z_{2}$ denotes the coordinate induced from the boundary class of the associated Maslov 0 disk, then the \textit{wall-crossing formula} is given as
\begin{equation*}\label{eqn:wcformualepsilon}
\begin{array}{l}
 z_{1} = z_{1}' (1+ T^{-\epsilon_{j}} z_{2}') \\
 z_{2} = z_{2}'
 \end{array}.
 \end{equation*}
The nontrivial monodromy of the affine structure near the nodal fiber (singularity of the affine structure) can be pushed to a single branch cut, which is usually taken be the ray from the singular fiber toward infinity.\\

\begin{remark}
Although collisions of walls result in complicated wall structures in general, the location of each blowup center can be carefully chosen (see \cite[Lemma 2.5]{HK23} for the precise convention) so that there exists a \textit{central chamber} $R$, surrounded only by initial walls (referring to those before scattering). Throughout, we fix our potential $W = W_{R}$ to be the potential valid over this central chamber. The constraint on the location of the blowups can be later lifted via analytic continuation of $W$ as in \cite{Yuan} (this is possible because the Lagrangian fibers are weakly unobstructed by \cref{lemma:unobstructedness}), and the potential can be thought of as a global function, at least away from a codimension $1$ region that cannot be covered (see \cref{rmk:critpointsoutside}).\\
\end{remark}

A \emph{broken disk} $\beta \in \mathrm{H}_{2}\left( \widetilde{X}, L_{u} \right)$ is a Maslov index $2$ pseudo-holomorphic disk belonging to the class
\[
	\qquad
	\widetilde{\beta}_{\nu_{j}} + \sum_{i=1}\widetilde{\beta}_{i}, 
	\qquad
	\widetilde{\beta}_{\nu_{j}},\, \widetilde{\beta}_{i} 
	\in \mathrm{H}_{2}\left( \widetilde{X}, L_{u} \right),
\]
where tildes are used to distinguish disks before the blowup process from their proper transformation after the blowup process. $\widetilde{\beta}_{\nu_{j}}$ is the Maslov index $2$ component of a broken disk, representing the proper transformation of the basic disk $\beta_{\nu_{j}}$ intersecting $D_{\Sigma, \nu_{j}}$ away from the non-toric blowup centers. $\widetilde{\beta}_{i}$'s are the Maslov index $0$ components, each representing the proper transformation of the disk precisely passing through a blowup point $p_{i}$ in $D_{\Sigma}$. It follows from \cref{eq:sympform} that the energy of a broken disk $\beta = \widetilde{\beta}_{\nu_{j}} + \sum_{i=1}\widetilde{\beta}_{i}$ is given by
\begin{equation*}\begin{aligned}
\widetilde{\omega }( \beta) 
	&= \widetilde{\omega}\left( \widetilde{\beta}_{\nu_{j}} \right) 
	+ \sum_{i=1} \widetilde{\omega}\left( \widetilde{\beta}_{i} \right)  \\
	&= \omega_{\Sigma}\left( \beta_{\nu_{j}} \right) + \sum_{i=1} \omega_{\Sigma}\left( \beta_{i} \right)
	- \sum_{i=1} \epsilon_{i}.
\end{aligned}\end{equation*}
Among such broken disks, \emph{basic broken disks} are those which break at parallel walls stemming from a single divisor $D_{\Sigma, j+1}$. More precisely, a basic broken disk is a broken disk of class
\[
	\widetilde{\beta}_{\nu_{j}} + l \widetilde{\beta}_{\nu_{j+1}}.
\]
We denote by $\mathcal{A}_{j}$ for the set consisting of such disks, arising from the two divisors $D_{\nu_{j}}$ and $D_{\nu_{j+1}}$. From \cref{eq:sympform}, the energy of such basic broken disks are given by 
\begin{equation*}
\omega_{\Sigma}\left( \beta_{\nu_{j}} \right) + l \omega_{\Sigma}\left( \beta_{\nu_{j+1}} \right) - \sum_{i=1} \epsilon_{i}.
\end{equation*}
From now on, we do not distinguish the disks $\beta_{\nu_{j}}$ in $X_{\Sigma}$ and their proper transformation $\widetilde{\beta}_{\nu_{j}}$ in $\widetilde{X}$ to avoid unnecessarily complicated notations.

\subsection{Critical points of the Landau-Ginzburg Model}

The relationship between critical points of Laurent polynomials and its tropicalization/Newton polytope is investigated in Section 4 of \cite{HK23}. We recall some basic definitions and results. Note that, although \cref{def:newtonpolytope} and \cref{Kushnirenko} are presented for general dimension $n$, we limit our discussion to the case $n=2$.\\

\begin{defn}\label{def:newtonpolytope}
Let $W =  \sum_{v \in \mathbb{Z}^{n}} \alpha_{v}z^{v} \in \Lambda[z_{1}^{\pm}, \ldots , z_{n}^{\pm}]$ be a Laurent polynomial with coefficients in the Novikov field.
\begin{enumerate}
\item
The \emph{Newton polytope} $\Delta_{W}$ of $W$ is defined as the convex hull of the support of $W$, 
\[
	\supp W := \{\,  v \in \mathbb{Z}^{n}   \;\, \rvert \;\,  \alpha_{v} \neq 0  \,\}.
\]

\item
A Laurent polynomial $W$ is said to be \emph{convenient} if the point $0 \in \mathbb{R}^{n}$ does not belong to any supporting plane of all $d$-dimensional faces  of $\Delta_{W}$ for $1\leq d \leq n-1$. 

\item
A Laurent polynomial $P$ is \emph{non-degenerate} if for any closed face $F$ of $\Delta_{P}$, the system 
\[
	\Bigl( z_{1}\frac{\partial W_F}{\partial z_{1}}\Bigl)_{F} = \dots  = \Bigl( z_{n}\frac{\partial W_F}{\partial z_{n}}\Bigl)_{F}=0
\]
has no solution in $(\Lambda^\times)^n$, where $W_{F} := \sum_{v \in F \cap \mathbb{Z}^{n}}\alpha_{v}z^{v}$ for a closed subset $F$ of $\mathbb{R}^{n}$.

\item
The \emph{tropicalization} of $W$ is a polyhedral complex in $\mathbb{R}^n$ (e.g., a piecewise linear graph on $\mathbb{R}^2$ when $n=2$) given as the corner locus of the piecewise linear function 
\[
	\tau_W : \mathbb{R}^n \to \mathbb{R}, \quad \,\,
	(x_{1},\cdots, x_{n}) \mapsto \min_{i} \, 
	\Bigl\{\,   \lambda_{i} + \left\langle v_{i} \, , \, (x_{1},\cdots,x_{n}) \right\rangle    \;\, \rvert \;\,  i \in \support W  \,\Bigr\}
\]
where $\lambda_i = \val (a_i)$. We denote the tropicalization of $W$ by $\mathrm{Trop}(W)$.

\item
The \emph{Newton subdivion} $\mathcal{S}_{W}$ of the Newton polytope $\Delta_{W}$ is the subdivision of the Newton polytope into lattice polygons, which is dual to the tropicalization $\mathrm{Trop}(W)$.

\item
A Laurent polynomial $W$ is said to be \emph{locally convenient} if any subcollection of monomials in $W$ determined by a cell in the subdivision $\mathcal{S}_W$ forms a convenient Laurent polynomial. 
\end{enumerate}
\end{defn}

\begin{thm}\cite[Theorem III]{Kush} \label{Kushnirenko}
Let $\mathbb{K}$ be an algebraically closed field with characteristic 0. If a convenient Laurent polynomial $W \in \mathbb{K}[z_{1}^{\pm}, \dots, z_{n}^{\pm}]$ is non-degenerate, then
\[
\rvert \mathrm{Crit}(W) \rvert = n!V_{n}(\Delta_{W})
\]
where $\rvert \mathrm{Crit}(W) \rvert$ is the number of critical points of $W$ counted with multiplicity, and $V_{n} (\Delta_W)$ is the $n$-dimensional volume of the Newton polytope $\Delta_W$.
\end{thm}

 Now suppose that $\alpha$ is a critical point of the potential $\widetilde{W}$ of $(\widetilde{X}, \widetilde{D})$, obtained from a toric model $\left( X_{\Sigma}, D_{\Sigma} \right)$ through a sequence of blowups. If $\mathrm{val}(\alpha) = (x_{0}, y_{0})$, we can write
\[
	\alpha = \left( T^{x_{0}}\underline{z}_{1}, T^{y_{0}}\underline{z}_{2} \right), 
	\qquad
	\underline{z}_{1}, \underline{z}_{2}  \in \Lambda_{U}
\]
and consider the restriction of $\widetilde{W}$ at the $\left( \Lambda_{U} \right)^{2}$-fiber over $\alpha$, which we denote by $\widetilde{W}_{\alpha}$. Then $\widetilde{W}_{\alpha}$ decomposes into the sum $W_{0} + W_{1}$, where $W_{0}$ consists of leading order terms of $\widetilde{W}_{\alpha}$, that is, terms with minimal valuation coefficients. Then by energy induction (see \cite[Theorem 10.4]{FOOO-T} or \cite[Theorem 4.37]{FW}), assuming that the $\mathrm{Hess}_{\alpha}(W_{0})$ is non-vanishing, each critical point of $W_{0}$ extends uniquely in a valuation-preserving manner to a critical point of $W$.

The tropicalization $\mathrm{Trop}(\widetilde{W})$ then confines the possible locations of critical points by determining how many terms $W_{0}$ has; if a point projects to a point $u \notin \mathrm{Trop}(\widetilde{W})$ under the valuation map, then $W_{0}$ consists of a single monomial, which attain the minimum valuation at $u$. Kushnirenko's\cref{Kushnirenko} then implies that $u$ cannot be a critical point of $\widetilde{W}$. Similarly, if $u$ lies on the interior of an edge $E$ of $\mathrm{Trop}(\widetilde{W})$, then $W_{0}$ contains two terms, and therefore $u$ cannot be a critical point of $\widetilde{W}$, provided that the edge $E$ is convenient.\\

\begin{remark}\label{rmk:nonconvenient}
The presence of non-convenient cells in the Newton subdivision poses several challenges, particularly in applying the energy induction on the leading order potential $W_{0}$. Fortunately, for generic choices of $\epsilon_{j}$ and $\lambda_{j}$, the Newton subdivision of the minimal potential (see \cref{rmk:minpotential}) is a locally convenient unimodal triangularization, apart from the possibility of two non-convenient $2$-cells sharing the unique edge containing the origin. For such non-convenient cells, detailed analysis of the leading order terms is carried out separately (see \cite[Lemma A.]{HK23}).\\ 
\end{remark}

Therefore, a critical point $\alpha$ of $\widetilde{W}$ always projects onto a vertex $V \in \mathrm{Trop}( \widetilde{W} )$ under $\mathrm{val} : \widetilde{X} \mathlarger{\mathlarger{\backslash}} \widetilde{D} \rightarrow B$, unless the Newton subdivision is non-convenient in which case $\alpha$ might map onto an edge $E \in \mathrm{Trop}( \widetilde{W} )$. In other words, for every critical point $\alpha$ of $\widetilde{W}$, there exists at least three (two if $\alpha$ projects onto a non-convenient edge) distinguished classes $\beta_{1}, \dots, \beta_{m \geq 3} \in H_{2}\big( \widetilde{X}, L_{\mathrm{val}(\alpha)} \big)$ supporting Maslov index $2$ disks which simultaneously attains minimum energy at $z= \alpha$. We refer to such disks as \emph{energy minimizing disks} of $\alpha$.

On the other hand, recall that we are only interested in \emph{geometric critical points}, which project onto a point in the interior of the moment polytope $\Delta$ of $\widetilde{X}$. We denote by $\mathrm{Crit}( \widetilde{W} )$ the set consisting of all geometric critical points of $\widetilde{W}$. The following Proposition classifies Maslov index $2$ disks based on whether they can become energy minimizers of a geometric critical point.

\begin{prop}[\cite{HK23} Proposition 4.2]\label{prop:basicbrokendisks}
Suppose that $(\widetilde{X}, \widetilde{D})$ is obtained from a toric model $\left( X_{\Sigma}, D_{\Sigma} \right)$ after a sequence of non-toric blowups. A critical point of $\widetilde{W}$ is geometric only if its minimizing disks are either basic disks, or basic broken disks. In particular, disks with sphere bubbles cannot be energy minimizing disks.
\end{prop}

\begin{remark}\label{rmk:minpotential}
\cref{prop:basicbrokendisks} allows us to consider only the \textit{minimal potential}, consisting of only the basic disks and basic broken disks;
\begin{equation}\label{eq:minpotential}
W_{\mathfrak{min}}
= \underbrace{\sum_{j=1}^{N}T^{\lambda_{j}}z^{\nu_{j}}}_{ W^{HV}_{\Sigma} }+ \sum_{j=1}^{N} \sum_{\pi^{\ast}(\beta) \in \mathcal{A}_{j}} N_{\beta} T^{\delta(\beta)}z^{\partial \beta}.
\end{equation}
Since a monomial of $W$ which is not contained in $W_{\mathfrak{min}}$ only produces non-geometric critical points, we do not distinguish $W$ and $W_{\mathfrak{min}}$ when there is no danger of confusion.\\
\end{remark}

\cref{prop:basicbrokendisks} is proven by solving inequalities arising from the energy minimizing disks. If an energy minimizing disk has ``excessive'' energy, its associated critical point must lie outside of the moment polytope. Moreover, using analogous arguments, it can be shown that the \textit{star-shaped Newton polytope}, formed by connecting adjacent toric lattices and $\mathcal{A}_{j}$'s without taking the convex hull, captures \textit{all} geometric critical point information of $\widetilde{W}$. For generic parameters, the Newton subdivision of the star-shaped Newton polytope of the (minimal) potential is a locally convenient unimodal triangulation \cite[Lemma 4.9]{HK23}, which we will assume throughout. Explicitly :

\begin{assumption}{Assumption}\label{assumption}
The (symplectic) parameters $0 < \epsilon_{j}\ll 1$ and $0 \ll \lambda_{j}$ are generic, in the sense that the Newton subdivision of the (minimal) potential is a locally convenient unimodal triangularization, apart from the central cell(s) containing the origin. Equivalently, every vertex of the tropicalization of the potential is trivalent with weight $1$.
\end{assumption}

Our discussion so far can be summarized as a complete criterion for geometric critical points :
\begin{prop}[\cite{HK23}]\label{prop:summary}
Let $(\widetilde{X}, \widetilde{D})$ be a log Calabi-Yau surface obtained via a sequence of blowups on the toric model $\left( X_{\Sigma}, D_{\Sigma} \right)$. Then for generic parameter, i.e. $\widetilde{\omega}$ satisfies \ref{assumption}, a critical point $\alpha$ of $\widetilde{W} = \widetilde{W}_{\mathfrak{min}}$ is a geometric critical point if and only if $\alpha$ lies over either
\begin{enumerate}[leftmargin=*]
\item
a vertex $V \in \mathrm{Trop}( \widetilde{W} )$ dual to a convenient $2$-cell of the Newton subdivision of the star-shaped Newton polytope, or
\item
an edge $E \in \mathrm{Trop}( \widetilde{W} )$ dual to the unique edge of the Newton subdivision that passes through the origin.
\end{enumerate}
Moreover, all geometric critical points are non-degenerate, that is, $\widetilde{W}$ is a Morse function. In particular, if $\alpha$ lies over a vertex $V \in \mathrm{Trop}( \widetilde{W} )$, then the dual $2$-cell $V^{\ast}$ is unimodal.
\end{prop}

\vspace{0.3cm}


\section{Blowdown of Log Calabi-Yau Surfaces}\label{sec:3}
We now focus on the log Calabi-Yau surface $(X, D)$ obtained by blowing down toric divisors of $(\widetilde{X}, \widetilde{D})$ with self-intersection number $-1$. Our primary interest lies in the change in the number of geometric critical points resulting from these blowdowns. Although we have already established some notations and conventions in \cref{sec:2}, we will reiterate and clarify them to ensure this section is self-contained and comprehensible.


\subsection{Notations and Geometric Setup}
Let $(\widetilde{X}, \widetilde{D})$ be a log Calabi-Yau surface obtained from the toric model $\left( X_{\Sigma}, D_{\Sigma} \right)$ after a sequence of non-toric blowups with generic sizes $\epsilon_{j}$, and suppose that $D_{\Sigma, i_{0}}$ is a toric divisor of $\widetilde{X}$ with self intersection number $(-1)$. After appropriate coordinate changes and renumbering of the divisors, we assume that $i_{0}=2$, and further fix the inward normals and the energy coefficients by
\begin{equation}\label{eqn:convention1}
\nu_{1} = (1,0), \;\; \nu_{2}=(0,1), \qquad \lambda_{1}=\lambda_{2}=0.
\end{equation}
Let us also denote by
\begin{equation}\label{eqn:convention2}
k = \text{ the number of non-toric blowup centers on } D_{\Sigma, 3}. 
\end{equation}
The (minimal) potential $\widetilde{W}$ comprises of terms associated with \textit{basic disks} $\beta_{\nu_{i}}$ and \textit{basic broken disks} $\beta \in \mathcal{A}_{i}$, where
\begin{equation}\label{eqn:basicbrokendisks}
	\mathcal{A}_{i} := 
	\left\{\, 
	\beta_{\nu_{i}} + l \beta_{\nu_{i+1}}
	\,\, \Big\rvert \,\,
	1 \leq l \leq 
	 \# \text{ of non-toric blowup centers in } D_{\Sigma, i+1}.
	\,\right\}
\end{equation}
Locally, the potential can be written as follows:
\begin{equation}\begin{aligned}\label{eq:potential_before}
\widetilde{W}
&= \underbrace{\sum_{j=1}^{N}T^{\lambda_{j}}z^{\nu_{j}}}_{ W^{HV}_{\Sigma} }+ \sum_{j=1}^{N} \sum_{\pi^{\ast}(\beta) \in \mathcal{A}_{j}} N_{\beta} T^{\delta(\beta)}z^{\partial \beta}  \\
&= z_{1} + z_{2} + \sum_{\pi^{\ast}(\beta) \in \mathcal{A}_{1}} N_{\beta} T^{\delta(\beta)}z^{\partial \beta} \,+\, \cdots\\
&= z_{1} + z_{2} + \sum_{i=1}^{k} T^{-\sum_{j=1}^{i}\epsilon_{j}}z_{1}z_{2}^{i}  \,+\, \cdots.
\end{aligned}\end{equation}
We denote by
\begin{equation}\label{eqn:geometriccriticalpoints}
\mathrm{Crit}\left( W \right) = 
\left\{ \, 
\alpha \in \left( \Lambda^{\times} \right)^{2} 
\;\;\, \bigg| \;\;\;
\mathrm{val}(\alpha) \in \Delta^{\circ}, \quad
\frac{\partial W}{\partial z_{1}}\big|_{\alpha} = \frac{\partial W}{\partial z_{2}}\big|_{\alpha} = 0
\, \right\}
\end{equation}
the set of all \textit{geomtric} critical points, projecting onto the interior of the moment polytope.\\

\subsection{The blowdown $(X, D)$ of $( \widetilde{X}, \widetilde{D} )$}

The blowdown process can be thought of as the reverse process of the blowup; that is, if $Z$ is a symplectically embedded $(-1)$-sphere in $( \widetilde{X}, \widetilde{\omega})$, the \emph{blowdown $(X, \omega)$ of $( \widetilde{X}, \widetilde{\omega} )$ along $Z$} is obtained by replacing a tubular neighborhood $\mathcal{N}(Z)$ of $Z$ with a solid ball $B(r)$ whose radius $r$ is determined by the symplectic form $\widetilde{\omega}$. The symplectic form $\omega$ is prescribed so that it is the standard symplectic form on the interior of the ball $B(r)$, while $\pi^{\ast}\omega = \widetilde{\omega}$ outside $\mathcal{N}(Z)$.

Alternatively, the blowdown process can be described as an example of the \textit{fiber connected sum}. We give a brief outline of the construction. Let $(M_{1}, \omega_{1})$ and $(M_{2}, \omega_{2})$ be two symplectic manifolds of the same dimension $2n$ and $(Q, \omega_{Q})$ be a compact symplectic manifold of dimension $2n-2$, with symplectic embeddings $\iota_{j} : Q \hookrightarrow M_{j}$, $j=1,2$. Let us denote by $N_{j}$ for the normal bundle of the image $Q_{j} = \iota_{j}(Q)$ in $M_{j}$ respectively. If the Euler classes $e(N_{1})$ and $e(N_{2})$ sum up to zero, the fiber connected sum $M_{1} \#_{Q} M_{2}$ of $M_{1}$ and $M_{2}$ along $Q$ is defined as follows: Let $\psi : N_{1} \rightarrow N_{2}$ be a fiber-orientation reversing bundle isomorphism, and let $V_{j}$ be the tubular neighborhood of $Q_{j}$ identified canonically with $N_{j}$. Then $M_{1} \#_{Q} M_{2}$ is given by
\begin{equation}\label{eqn:fiberconnectedsum}
M_{1} \#_{Q} M_{2} : =
\left( M_{1} \mathlarger{\mathlarger{\backslash}} Q_{1} \right)
\mathlarger{\cup}_{\varphi}
\left( M_{2} \mathlarger{\mathlarger{\backslash}} Q_{2} \right).
\end{equation}
Here, the gluing map $\varphi  : V_{1} \mathlarger{\mathlarger{\backslash}} Q_{1} \rightarrow V_{2} \mathlarger{\mathlarger{\backslash}} Q_{2}$ is the composition $\varphi = f \circ \psi $ where $f$ is a diffeomorphism that flips each punctured normal fiber inside out.

Returning back to our case in dimension $4$, we have $(M_{1}, \omega_{1}) = ( \widetilde{X}, \widetilde{\omega} )$, $(M_{2}, \omega_{2}) = \left( \mathbb{P}^{2}, \omega_{FS} \right)$, and $Z_{1} = D_{\Sigma, 2}$ with self-intersection $(-1)$. The normal bundle $N_{1}$ of $Z_{1} \subset \widetilde{X}$ is the tautological line bundle $L$, whereas the normal bundle $N_{2}$ of a symplectically embedded $(-1)$-sphere $Z_{2} \subset \mathbb{P}^{2}$ is the dual $L^{\ast}$. After rescaling $\omega_{FS}$ appropriately so that the symplectic areas of $Z_{1}$ and $Z_{2}$ match, the fiber connected sum $\widetilde{X} \#_{Z} \mathbb{P}^{2}$ is unique up to symplectomorphism. From our convention \eqref{eqn:convention1}, the cohomology class of $\omega$ is simply given by 
\begin{equation*}\begin{aligned}
[\omega]
&= \sum_{i \neq 1} \lambda_{i}\pi^{\ast}[D_{\Sigma,i}]^{\vee} - \sum_{j=1}^{k} \epsilon_{j}[E_{j}]^{\vee}\\
&=\sum_{i = 3}^{N} \lambda_{i}\pi^{\ast}[D_{\Sigma,i}]^{\vee} - \sum_{j=1}^{k} \epsilon_{j}[E_{j}]^{\vee} = [\omega],
\end{aligned}\end{equation*}
that is, the class of the symplectic form remains unchanged after the blowdown process.

This construction allows us to easily define a special Lagrangian torus fibration on the blowdown $X = \widetilde{X} \#_{Z} \mathbb{P}^{2}$. In \cite{BCHL}, a special Lagrangian torus fibration on $( \widetilde{X} \mathlarger{\mathlarger{\backslash}} \widetilde{D} )$ (with respect to the  K\"{a}hler form $\widetilde{\omega}$ which is equal to the pull back of $\omega_{\Sigma}$ outside a small neighborhood $U_{j}$ of each exceptional divisor $E_{j}$, while coinciding with an $S^{1}$-invariant local  K\"{a}hler form obtained by averaging over the $S^{1}$-action on $U_{j}$) is constructed. Locally, the fibration is given by 
\[
	(x, y) \mapsto \left( \lvert x \rvert, \mu_{S^{1}} (x, y) \right) \in \mathbb{R}^{2},
\]
where $\mu_{S^{1}}$ is the moment map of the for the $S^{1}$-action $(x,y) \mapsto (x, e^{i\theta}y)$. Similarly, we have  
\[
	(x', y') \mapsto \left( \lvert x' \rvert, -\mu_{S^{1}} (x', y') \right) \in \mathbb{R}^{2}
\]
for $(x', y') \in \mathbb{P}^{2} \mathlarger{\mathlarger{\backslash}} Z_{2}$. The two fibrations are then smoothly glued along their deleted tubular neighborghood $V_{1} \mathlarger{\mathlarger{\backslash}} Z_{1}$ and $V_{2} \mathlarger{\mathlarger{\backslash}} Z_{2}$ via $\varphi$. In the process, one should modify the (rescaled) symplectic form $\omega_{FS}$ on $\mathbb{P}^{2}$ so that it fiber-wise matches with the symplectic form $\widetilde{\omega}$ on $\widetilde{X}$ if necessary. See \cite{Gompf} for a precise description of the construction of the symplectic form $\widetilde{\omega}$ on fiber connected sums.

\subsection{Weakly unobstructedness of $L_{u}$}

In order to make sense of the potential $W$ of $(X,D)$ after blowdown, we first show that Lagrangian fibers $L_{u}$ are weakly unobstructed :

\begin{lemma}\label{lemma:unobstructedness}
Let $(X, D)$ be a log Calabi-Yau surface. Then the Lagrangian fibers $L_{u}$ are weakly unobstructed, that is, for any $b \in H^{1}(L_{u}; \Lambda_{+})$, 
\[
	m_{0}(1) + m_{1}(b) + m_{2}(b, b) + \cdots = W(b) \cdot [L_{u}].
\]
\end{lemma}

\begin{proof}
The virtual dimension of $\mathcal{M}_{k}(X, L_{u}, \beta)$ is given by
\begin{equation}\label{eqn:virdim}
\mathrm{dim} \left( L_{u} \right)  + \mu(\beta) + k -3 = \mu(\beta) + k -1.
\end{equation}
If $k=0$, then $\mu(\beta) + k -1 < 0$ for $\mu(\beta) \leq 0$, implying that the moduli space $\mathcal{M}_{0}(X, L_{u}, \beta)$  is empty after perturbing the Kuranishi structure when $\mu(\beta) \leq 0$. Hence for $\beta$ with $\mu(\beta) \leq 0$, the virtual fundamental chain of $\mathcal{M}_{0}(X, L_{u}, \beta)$ can be taken to be zero, and using compatibility of the forgetful map
\[
	\mathfrak{forget} : \mathcal{M}_{1}(X, L_{u}, \beta) \rightarrow \mathcal{M}_{0}(X, L_{u}, \beta)
\]
with the Kuranishi structure (\cite[Theorem 3.1, Corollary 3.1]{Fuk-Cyc}), the virtual fundamental chain of $\mathcal{M}_{k}(X, L_{u}, \beta)$ can also be taken to be zero after Kuranishi perturbation for all $k \geq 0$ when $\mu(\beta) \leq 0$.
For disks with strictly positive Maslov index, the standard degree argument shows that the degree of $m_{k, \beta} (b, \dots, b) = 2- \mu(\beta)$ is non-negative only if $\mu(\beta) = 2$, in which case is zero, i.e. a multiple of $[L_{u}]$.
\end{proof}

\begin{remark}\label{rmk:critpointsoutside}
Note that \cref{lemma:unobstructedness} holds due to dimension reasons, and is true for any compact surface $X$, as long as \cite[Theorem 3.1, Corollary 3.1]{Fuk-Cyc} is applicable.
The proof of \cref{lemma:unobstructedness} shows that Maslov index $0$ disks disappear after Kuranishi perturbation. That is, log Calabi-Yau surfaces satisfy \cite[Assumption 1.2]{Yuan}, and \cite[Theorem 1.3]{Yuan} gives us a rigid analytic variety $\check{Y}$ along with a global mirror LG model $W$, obtained by gluing local charts via wall-crossing formula. We can then fix an expression of $W$ at a single chamber, and simply enlarge the domain through analytic continuation.

Although there exists a codimension $1$ region which cannot be covered by analytic continuation, following \cite{HKL} and \cite{HK23}, we can introduce additional mirror charts induced by the Floer deformation space of the nodal fiber, and associate critical points lying outside the initial chamber (but within the moment polytope) with geometric objects.

\end{remark}

\subsection{Number of geometric critical points}
We now investigate the change in the number of geometric critical points as we blow down a toric divisor $D_{\Sigma, 2}$ with self-intersection number $(-1)$. Blowing down $D_{\Sigma,2}$, all Maslov index $2$ disks stemming from $D_{\Sigma, 2}$ (including the basic disk $\beta_{\nu_{2}}$ and basic broken disks in $\mathcal{A}_{2}$) become Maslov index $0$, and thus no longer contribute to the potential. According to the Kushnirenko theorem \cite{Kush}, this results in the loss of multiple critical points, which seems to contradict the fact that the rank of the quantum cohomology decreases by only one.

The key observation, as outlined in \cref{prop:main}, is that broken disks of Maslov index $2$ stemming from $D_{\Sigma, 1}$ - which did not contribute to the potential prior to blowdown due to excessive energy - effectively compensates for all but one of the lost disks. Moreover, the Kushnirenko theorem alone is insufficient in determining the precise number of \textit{geometric} critical points, as some critical points may lie outside of the moment polytope. \cref{prop:summary} tells us that in order to capture only the \textit{geometric} critical point information of $W$, the usual convex Newton polytope should be replaced with the \textit{star-shaped Newton polytope}. See \cite{HK23} for the precise description of the star-shaped Newton polytope. While \cref{prop:summary} applies directly when $D_{\Sigma, 2}$ does not support any non-toric blowup center, this is no longer the case if $D_{\Sigma, 2}$ does indeed support non-toric blowup centers, as the broken disks compensating for the lost disks are no longer \textit{basic broken disks}.

\begin{remark}
We make use of the holomorphic-tropical correspondence by choosing appropriate Kuranishi structures as established in various works (e.g. \cite{GrPS}, \cite{CPS}, \cite{Lin}, \cite{Lin24}, \cite{BCHL}), and do not distinguish the holomorphic disk count $N_{\beta}$ and the tropical disk count $N_{\beta}^{\mathrm{trop}}$. Broken disks are represented as \textit{broken lines} in \cref{fig:case1} and \cref{fig:case2}. See \cite{CPS} and \cite{BCHL} for a concrete definition, or \cite{HK23} for a short summary.
\end{remark}

\begin{prop}\label{prop:main}
Let $(\widetilde{X}, \widetilde{D})$ be a log Calabi-Yau surface obtained by performing a sequence of blowups on a toric model $\left( X_{\Sigma}, D_{\Sigma} \right)$, and let $(X, D)$ be the log Calabi-Yau surface after blowing down $m$ toric divisors of $(\widetilde{X}, \widetilde{D})$ with self-intersection number $-1$. Then
\[
	\left| \mathrm{Crit} \left( W \right)\right| = \big| \mathrm{Crit} \left( \widetilde{W} \right)\big| -m .
\]
Moreover, for generic parameters, $W$ is Morse if $\,\widetilde{W}$ is Morse.
\end{prop}

\begin{proof}
We first assume $m=1$, and show that the number of geometric critical points decreases precisely by one when contracting $D_{\Sigma, 2}$. The following arguments follow inductively for when $m \geq 2$. We subdivide the blowdown process into three scenarios. The first case will be when $D_{\Sigma,2}$ supports no blowup center in its interior, while the second and third cases are when $D_{\Sigma,2}$ indeed does support one or more non-toric blowup center.

\emph{(Case I.)} Suppose that $D_{\Sigma,2}$ supports no blowup center. Then we have $\nu_{2}=\nu_{1} + \nu_{3}$, and hence $\nu_{3} = (-1,1)$. The potential is given by
\begin{equation}\begin{aligned}\label{eq:potential_before1}
\widetilde{W}	&= \sum_{i = 1} T^{\lambda_{i}}z^{\nu_{i}} + \sum_{\pi^{\ast}(\beta) \in \mathcal{B}} N_{\beta}T^{\delta(\beta)}z^{\partial \beta},
\qquad
\mathcal{B} := \bigcup_{j=1}^{N} \mathcal{A}_{j}.  \\
	&= z_{1} + z_{2} + T^{\lambda_{3}} \frac{z_{2}}{z_{1}}
	+ \sum_{l=1}^{k} \overbrace{\left(z_{2}\right) \left(T^{\lambda_{3}}\frac{z_{2}}{z_{1}}\right)^{l} }^{\in \mathcal{A}_{2}} + \cdots
\end{aligned}\end{equation}
where we set $\epsilon_{i} = 0$ for notational simplicity (see \cref{rmk:epsilon}).

Blowing down $D_{\Sigma,2}$, the disks $\beta_{\nu_{2}} + j \beta_{\nu_{3}} \in \mathcal{A}_{2}$ (represented in blue in \cref{fig:case1_before}) disappear and no longer contribute to the potential. Instead, the broken disks $\beta_{\nu_{1}} + j \beta_{\nu_{3}}$ stemming from $D_{\Sigma,1}$ now contribute to the (minimal) potential. Note that these disks do contribute to the full potential $\widetilde{W}$ prior to blowdown, but are \textit{not} leading order terms, as there are monomials which represent same disk classes with smaller energy. Let us denote by
\[
	\mathcal{A}'_{2} := 
	\left\{\, 
	\beta_{\nu_{1}} + l \beta_{\nu_{3}}
	\,\, \Big\rvert \,\,
	1 \leq l \leq 
	k =  \# \text{ of non-toric blowup centers in } D_{\Sigma, 3}
	\,\right\}
\]
for the set consisting of these new basic broken disks, depicted in red in \cref{fig:case1_after}.

\begin{figure}[h]
\centering
\subcaptionbox{%
	Local model of $\left(\widetilde{X}, \widetilde{D}\right)$ when $D_{\Sigma,2}$ supports no non-toric blowup centers. Basic broken disks from $\mathcal{A}_{2}$ which become Maslov index $0$ after blowdown are represented in blue.
	\label{fig:case1_before}%
	}
	{\makebox[0.45\linewidth][c]
		{\begin{tikzpicture}[scale=0.09]
				
		\coordinate (A) at (0,0);
		\coordinate (B) at (0,50);
		\coordinate (C) at (40,0);
		\coordinate (D) at (70,30);
		\coordinate (E) at (0,-40);		
		\node[circle, opacity=0] (a) at (0,-20) {};

		\draw[line width = 1pt] (B) -- (A) -- (C) -- (D);
		\draw[line width =0.7pt] (41,1) node[cross, rotate = 45] {};
		\draw[line width =0.7pt] (44,4) node[cross, rotate = 45] {};
		\draw [dashed, line width=0.7pt] (41,1)--(-5,47);
		\draw [dashed, line width=0.7pt] (44,4)--(-5,53);
		\draw (46,-1) node[left, rotate = -45] { \huge\} };
		\draw (45,-1) node[right] { k };
		
		\draw (0,20) node[left] { $D_{\Sigma,1}$ };
		\draw (20,0) node[below] { $D_{\Sigma,2}$ };				
		\draw (60,15) node[below] { $D_{\Sigma,3}$ };
				
		\draw [line width=0.8pt, color=blue!70] (34,0)--(34,8)--(19,37);		\draw [line width=0.8pt, color=blue!70] (19,0)--(19,37);
		
		\draw (19,15) node[left] { $\beta_{\nu_{2}}$ };
		\draw (23,31) node[right] { $\beta_{\nu_{2}} + \beta_{\nu_{3}}$ };

		\filldraw (19,37) circle (10pt);		
		\end{tikzpicture}%
	}
	}\hfill%
\subcaptionbox{%
	Local model of $(X, D)$ after contracting $D_{\Sigma,2}$. Broken disks which contribute to the potential after blowdown are shown in red.
	\label{fig:case1_after}%
	}
	{\makebox[0.45\linewidth][c]	
		{\begin{tikzpicture}[scale=0.07]
		
		\coordinate (A) at (0,0);
		\coordinate (B) at (0,50);
		\coordinate (C) at (40,0);
		\coordinate (D) at (70,30);
		\coordinate (E) at (0,-40);		
		
		\draw[line width = 1pt] (B) -- (E);
		\draw[line width = 1pt] (E) -- (D);
		\draw[dotted, line width = 0.7pt] (A) -- (C);				
		\draw[line width =0.7pt] (41,1) node[cross, rotate = 45] {};
		\draw[line width =0.7pt] (44,4) node[cross, rotate = 45] {};
		\draw [dashed, line width=0.7pt] (41,1)--(-5,47);
		\draw [dashed, line width=0.7pt] (44,4)--(-5,53);
		\draw (46.5,-1.5) node[left, rotate = -45] { \huge\} };
		\draw (44.5,-1.5) node[right] { k };
		
		\draw (21,31) node[right] { $\beta_{\nu_{1}} + \beta_{\nu_{3}}$ };		
		
		\draw [line width=0.8pt, color=red!70] (19,37)--(19,23)--(0,23);										
		\filldraw (19,37) circle (10pt);		
		
		\end{tikzpicture}
		}
	}%
\caption{\emph{(Case I.)}}
\label{fig:case1}
\end{figure}

The potential after blowing down $D_{\Sigma,2}$ becomes
\begin{equation}\begin{aligned}\label{eq:potential_after1}
W	&= z_{1} + T^{\lambda_{3}}\frac{z_{2}}{z_{1}} 
	+ \sum_{l=1}^{k} \overbrace{\left(z_{1}\right)\left(T^{\lambda_{3}}\frac{z_{2}}{z_{1}}\right)^{l}}^{
	\in \mathcal{A}'_{2}
	}
	+ \cdots\\
	& = z_{1} + T^{\lambda_{3}} z_{2} + T^{\lambda_{3}}\frac{z_{2}}{z_{1}}
	+ \sum_{l=1}^{k} \left(z_{2}\right)\left(T^{\lambda_{3}}\frac{z_{2}}{z_{1}}\right)^{l}
	+ \cdots.
\end{aligned}\end{equation}
Notice that since $\nu_{2}=\nu_{1} + \nu_{3}$, the broken disk $\beta_{\nu_{1}} + l\beta_{\nu_{3}}$ plays precisely the role of $\beta_{\nu_{2}} + (l-1) \beta_{\nu_{3}}$, just with greater energy (by $\lambda_{N}$ to be precise).

Since the disks $\beta_{\nu_{1}} + j \beta_{\nu_{3}} \in \mathcal{A}'_{2}$ are \textit{basic} broken disks, it follows from \cref{prop:summary} (see also \cref{rmk:toricblowdown} below) that all geometric critical points of $W$ corresponds to $2$-cell of the star-shaped Newton polytope. Comparing \cref{eq:potential_before1} and \cref{eq:potential_after1},
\[
	\support \widetilde{W}
	=
	\support W
	\cup
	\left\{\,
	(-k,k+1)
	\,\right\},
\]
and therefore it follows that $W$ has precisely one less geometric critical point than $\widetilde{W}$.\\

\begin{remark}\label{rmk:epsilon}
As we have set $\epsilon_{j} = 0$ for notational simplicity,  all non-toric critical points lie on the corner $D_{\Sigma,1} \cap D_{\Sigma,3}$. When exceptional divisors are given generic positive small enough sizes, non-toric critical points lie arbitrarily close to the corner, in the interior of the moment polytope.\\
\end{remark}

\begin{remark}\label{rmk:toricblowdown}
It is worth noting that even though $(X, D)$ is obtained from $( \widetilde{X}, \widetilde{D} )$ by performing blowdown, $(X, D)$ in this case (where $D_{\Sigma, 2}$ supports no non-toric blowup center) is a case which is essentially already covered by \cref{prop:summary}. More precisely, since the blowdown process is completely toric, $(X, D)$ can be understood as a surface obtained after the exact same sequence of blowups from the toric model $(X_{\widetilde{\Sigma}}, D_{\widetilde{\Sigma}})$, \textit{after} contracting $D_{\Sigma, 2}$ from $\left( X_{\Sigma}, D_{\Sigma} \right)$.
Although this perspective gives a short proof for \emph{(Case I.)}, we have detailed the steps of tracking each lost and gained disk, as this serves as an illustrative example for the latter two cases discussed below.\\
\end{remark}

\emph{(Case II.)} Now suppose that $D_{\Sigma,2}$ supports one non-toric blowup center with symplectic size $\epsilon$. That is, $D_{\Sigma,2}$ before the non-toric blowup is a curve with self-intersection $0$, and hence $\nu_{3}=(-1,0)$. The potential is given by
\begin{equation}\begin{aligned}\label{beforeblowdown1}
\widetilde{W} 	&= z_{1} + z_{2} + T^{-\epsilon}z_{1}z_{2} + \frac{T^{\lambda_{3}}}{z_{1}} 
	+\sum_{\pi^{\ast}(\beta) \in \mathcal{A}_{2}}N_{\beta}T^{\delta(\beta)}z^{\partial \beta} + \cdots \\
	&= z_{1} + z_{2} + T^{-\epsilon}z_{1}z_{2} + \frac{T^{\lambda_{3}}}{z_{1}} 
	+ \sum_{j=1}^{k} z_{2} \left( \frac{T^{\lambda_{3}}}{z_{1}} \right)^{j} + \cdots \\
	&= z_{1} + z_{2} + T^{-\epsilon}z_{1}z_{2} + \frac{T^{\lambda_{3}}}{z_{1}}
	+ \sum_{j=1}^{k}T^{j\lambda_{3}}\frac{z_{2}}{z_{1}^{j}} + \cdots,
\end{aligned}\end{equation}
where we have set the sizes of the non-toric blowups on $D_{\Sigma,3}$ to be zero for notational convenience. See \cref{fig:case2_before}, where some of the disks contributing to the potential are shown in blue.

\begin{figure}[h!]
\centering
\subcaptionbox{%
	Local model of $( \widetilde{X}, \widetilde{D} )$ when $D_{\Sigma,2}$ supports one non-toric blowup center. The shaded region $R_{1}$ is where the affine coordinate change is applied.
	\label{fig:case2_before}%
	}
	{\makebox[0.8\linewidth][c]
		{\begin{tikzpicture}[scale=0.11]
				
		\coordinate (A) at (0,0);
		\coordinate (B) at (0,50);
		\coordinate (a) at (10,0);
		\coordinate (b) at (10,10);
		\coordinate (c) at (20,0);
		\coordinate (d) at (72,0);
		\coordinate (e) at (92,0);		
		\coordinate (C) at (100,0);
		\coordinate (D) at (100,30);
		\coordinate (E) at (0,-40);		
		
		\fill[blue!20, draw=none, opacity=0.3] (10,51) -- (10,10) --(27,-7)-- (85,-7)--(85,50);
		
		\draw[line width = 1pt] (B) -- (0,10) -- (63,10);
		\draw[dotted, line width = 1pt] (63,10) -- (80,10);
		\draw[line width = 1pt] (80,10)--(85,10)--(85,30);		
										
		\draw[dashed, line width = 1pt] (10,-7)--(a) -- (b) -- (c)--(27,-7);
		\fill[green!30, draw=none, opacity=0.4] (10,-7) -- (b) -- (27,-7);
		
		\draw [line width = 0.8pt, >=latex, <->] (8,4.3) to [out=-50,in=-130] (18,4.5);		
		\draw (15.6,1.6) node[below] { \small{$\left(\begin{matrix} 1 & 1 \\ 0 & 1 \end{matrix}\right)$ }};

		\draw[line width =0.7pt] (10,10) node[cross] {};
		\draw[dashed, line width = 0.7pt] (10,10)--(10,51);

		\draw[line width =0.7pt] (85,15) node[cross] {};
		\draw[line width =0.7pt] (85,18) node[cross] {};		
		\draw[dashed, line width = 0.7pt] (85,15)--(80,15);
		\draw[dotted, line width = 0.7pt] (80,15)--(63,15);		
		\draw[dashed, line width = 0.7pt] (63,15)--(-3,15);	
			
		\draw[dashed, line width = 0.7pt] (85,18)--(80,18);
		\draw[dotted, line width = 0.7pt] (80,18)--(63,18);		
		\draw[dashed, line width = 0.7pt] (63,18)--(-3,18);		
		
		\draw (0,44) node[left] { $D_{\Sigma,1}$ };
		\draw (80,10) node[below] { $D_{\Sigma,2}$ };
		\draw (85,27) node[left] { $D_{\Sigma,3}$ };
		\draw (64,40) node[left] { $R_{1}$ };

		\draw [line width=0.8pt, color=blue!70] (40,0)--(30,10)--(30,40);
			\draw (24,34.6) node[left] { \small{$\beta_{\nu_{1}} + \beta_{\nu_{2}}$ }};
		\draw [line width=0.8pt, color=blue!70] (0,40)--(30,40);
			\draw (20,40) node[above] { \small{$\beta_{\nu_{1}}$ }};				
		\draw [line width=0.8pt, color=blue!70] (30,40)--(10,20)--(0,20);
			\draw (30,27) node[right] { \small{$\beta_{\nu_{2}}$ }};				
		\draw [line width=0.8pt, color=blue!70] (65,0)--(55,10)--(55,15)--(30,40);
			\draw (44,27) node[right] { \small{$\beta_{\nu_{2}} + \beta_{\nu_{3}}$ }};				

		\filldraw (30,40) circle (12pt);

		\end{tikzpicture}%
	}
	}\hfill%
\subcaptionbox{%
	Local model of $( \widetilde{X}', \widetilde{D}' )$ after applying affine coordinate change.
	\label{fig:case2_coor}%
	}
	{\makebox[0.8\linewidth][c]	
		{\begin{tikzpicture}[scale=0.1]
		
		\coordinate (A) at (0,0);
		\coordinate (B) at (0,60);
		\coordinate (a) at (10,0);
		\coordinate (b) at (10,10);
		\coordinate (c) at (20,0);
		\coordinate (d) at (64,0);
		\coordinate (e) at (77,0);
		\coordinate (C) at (80,0);
		\coordinate (D) at (120,40);
		\node[circle, opacity=0] (0) at (-20,64) {};	
		
		\draw[line width = 1pt] (B) -- (0,10) -- (64,10);		
		\draw[dotted, line width = 1pt] (64,10) -- (77,10);
		\draw[line width = 1pt] (77,10) -- (90,10) -- (D);					
				
		\draw[dotted, line width = 1pt] (a) -- (b);
		\fill[green!30, draw=none, opacity=0.4] (b) -- (10,61) -- (61,61);

		\draw[line width =0.7pt] (10,10) node[cross] {};
		\draw[dashed, line width = 0.7pt] (10,10)--(10,61);
		\draw[dashed, line width = 0.7pt] (10,10)--(61,61);		
		
		\draw[line width =0.7pt] (92,12) node[cross, rotate = 45] {};
		\draw [dashed, line width=0.7pt] (92,12)--(52,52);
		\draw [dotted, line width=0.7pt] (52,52)--(10,52);
		\draw [dashed, line width=0.7pt] (10,52)--(-3,52);		
		
		\draw[line width =0.7pt] (93.5,13.5) node[cross, rotate = 45] {};
		\draw [dashed, line width=0.7pt] (93.5,13.5)--(53.5,53.5);
		\draw [dotted, line width=0.7pt] (53.5,53.5)--(10,53.5);
		\draw [dashed, line width=0.7pt] (10,53.5)--(-3,53.5);

		\draw [line width=0.8pt, color=blue!70] (58,58)--(76,40);
		\draw [dotted, line width=0.7pt, color=blue!70] (58,58)--(10,58);
		\draw [line width=0.8pt, color=blue!70] (10,58)--(0,58);
			\draw (67,50) node[right] { \small{$\beta_{\nu_{1}'}$} };

		\draw [line width=0.8pt, color=blue!70] (40,40)--(76,40);
		\draw [dotted, line width=0.7pt, color=blue!70] (40,40)--(10,40);
		\draw [line width=0.8pt, color=blue!70] (10,40)--(0,40);
			\draw (55,40) node[below] { \small{$\beta_{\nu_{1}'} + \beta_{\nu_{2}'} $}};

		\draw [line width=0.8pt, color=blue!70] (76,40)--(85,22);
		\draw [line width=0.8pt, color=blue!70] (85,22)--(85,10);
			\draw (80,32) node[right] { \small{$\beta_{\nu_{2}'} + \beta_{\nu_{3}'} $}};

		\draw [line width=0.8pt, color=blue!70] (76,40)--(76,10);
			\draw (76,18) node[left] { \small{$\beta_{\nu_{2}'} $}};
							
		\filldraw (76,40) circle (12pt);

		\end{tikzpicture}
		}
	}%
\caption{}
\label{fig:case2}
\end{figure}

In order to keep track of the potential and the moment polytope as we blow down $D_{\Sigma,2}$, we apply an affine coordinate change to the region on the right side of the wall, denoted as $R_{1}$, so that the three adjacent toric divisors align as in \emph{(Case I.)}. The region $R_{1}$ is depicted in blue in \cref{fig:case2_before}. Note that applying affine coordinate changes to the symplectic affine structure is nothing but changing the chosen basis $\{ f_{1}, f_{2} \}$ of $H_{1}(L_{u})$ in \cref{eqn:mirrorzi}.

Recall that the monodromy of the symplectic affine structure around the nodal fiber (in terms of $z_{i}$-coordinates) is given by the matrix $\left( \begin{matrix} 1 & 0 \\ 1 & 1 \end{matrix}\right)$. We apply the \textit{inverse} matrix $\left( \begin{matrix} 1 & 0 \\ -1 & 1 \end{matrix}\right)$ on $R_{1}$, so that the monodromy of symplectic affine coordinates becomes concentrated along the wall \textit{above} the nodal fiber. Let us denote by $\nu_{j}'$ for the inward normals after our coordinate change. We have $\nu_{1}' = (1,-1) = \nu_{1} - \nu_{2}$ and $\nu_{2}' = (0,1) = \nu_{2}$, and hence the basic disks after the coordinate change are given by $\beta_{\nu_{1}'} = \beta_{\nu_{1}} - \beta_{\nu_{2}}$ and $\beta_{\nu_{2}'} = \beta_{\nu_{2}}$.  See \cref{fig:case2_coor}. Therefore the potential function is given by:
\begin{equation}\begin{aligned}\label{eq:case2-before2}
\widetilde{W}' = \widetilde{W}'(z_{1}', z_{2}') &= z_{1}' + z_{2}' + T^{-\epsilon}z_{1}'z_{2}' + \frac{T^{\lambda_{3}}}{z_{1}'}
+ \sum_{l=1}^{k}\left(z_{2}'\right)\left(\frac{T^{\lambda_{3}}}{z_{1}'}\right)^{l} + \cdots \\
&= \frac{z_{1}}{z_{2}} + z_{2} + T^{-\epsilon}z_{1} + T^{\lambda_{3}}\frac{z_{2}}{z_{1}}
+ \sum_{l=1}^{k} \left(z_{2}\right)\left(T^{\lambda_{3}}\frac{z_{2}}{z_{1}}\right)^{l} + \cdots.
\end{aligned}\end{equation}

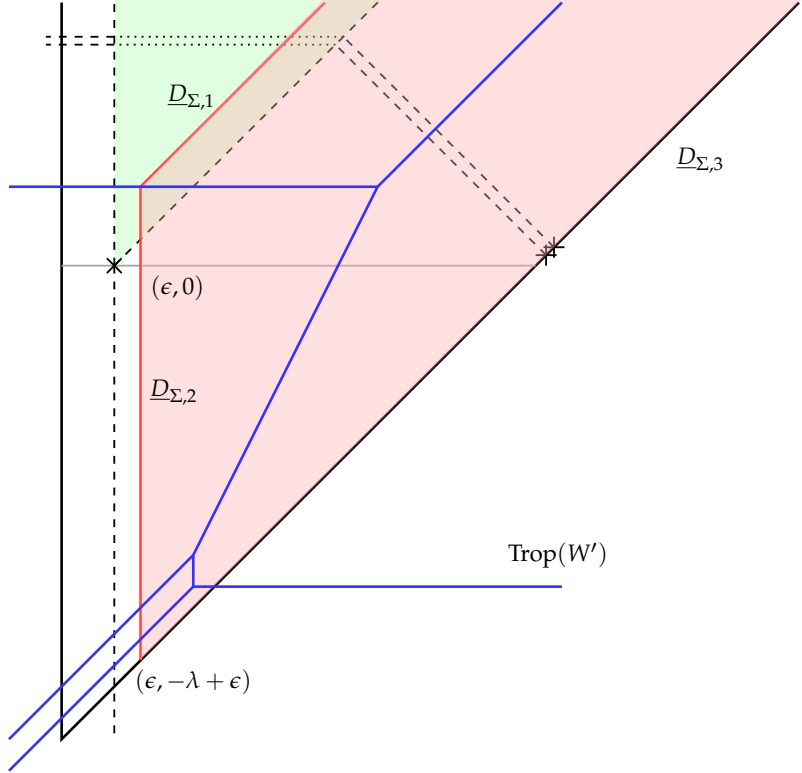
\begin{figure}[h!]
\centering
\begin{tikzpicture}[scale=0.06]

		\coordinate (A) at (0,0);
		\coordinate (B) at (0,60);
		\coordinate (a) at (10,0);
		\coordinate (b) at (10,10);
		\coordinate (c) at (20,0);
		\coordinate (d) at (64,0);
		\coordinate (e) at (77,0);
		\coordinate (C) at (80,0);
		\coordinate (D) at (140,60);
		\node[circle, opacity=0] (0) at (-20,64) {};	
		
		\draw[line width = 1pt] (B) -- (0,-80) -- (D);
		\draw[line width = 0.7pt, color=gray!70] (0,10) -- (90,10);				
			
		\draw[dashed, line width = 0.7pt] (10,10) -- (10,-80);
		\fill[green!30, draw=none, opacity=0.4] (b) -- (10,61) -- (61,61);

		\draw[line width =0.7pt] (10,10) node[cross] {};
		\draw[dashed, line width = 0.7pt] (10,10)--(10,61);
		\draw[dashed, line width = 0.7pt] (10,10)--(61,61);		
		
		\draw[line width =0.7pt] (92,12) node[cross, rotate = 45] {};
		\draw [dashed, line width=0.7pt] (92,12)--(52,52);
		\draw [dotted, line width=0.7pt] (52,52)--(10,52);
		\draw [dashed, line width=0.7pt] (10,52)--(-3,52);		
		
		\draw[line width =0.7pt] (93.5,13.5) node[cross, rotate = 45] {};
		\draw [dashed, line width=0.7pt] (93.5,13.5)--(53.5,53.5);
		\draw [dotted, line width=0.7pt] (53.5,53.5)--(10,53.5);
		\draw [dashed, line width=0.7pt] (10,53.5)--(-3,53.5);

		\draw[line width = 1pt, color= red!80] (50,60)--(15,25) -- (15,-65);
		\fill[red!30, draw=none, opacity=0.4] (50,61) -- (15,25) -- (15,-65)--(141,61);
			\draw (31,42) node[left] { \small{$\underline{D}_{\Sigma,1}$} };
			\draw (15,-14) node[right] { \small{$\underline{D}_{\Sigma,2}$} };					
			\draw (115,30) node[right] { \small{$\underline{D}_{\Sigma,3}$} };
			
			\draw (22,10) node[below] { \small{$(\epsilon,0)$} };
			\draw (25,-65) node[below] { \small{$(\epsilon,-\lambda + \epsilon)$} };			

			\draw (83,-45) node[right] { \small{$\mathrm{Trop}(W')$} };

		\draw[line width = 1pt, color= blue!80] (-10,-86)--(25,-51)--(25,-45)--(60,25)--(-10,25);
		\draw[line width = 1pt, color= blue!80] (-10,-80)--(25,-45);
		\draw[line width = 1pt, color= blue!80] (25,-51)--(95,-51);
		\draw[line width = 1pt, color= blue!80] (60,25)--(95,60);

\end{tikzpicture}
\caption{ $(X', D' )$ and $(\underline{X}, \underline{D})$. The moment polytope of $(\underline{X}, \underline{D})$ is shown in red. The tropicalization of $W'$ is shown in blue.}%
 \label{fig:case2_toric}
\end{figure}

We can now perform the ``typical'' toric blowdown. Blowing down $D_{\Sigma,2}$, the disks $\beta_{2}' + l \beta_{3}'$, $l =0, \dots, k$ become Maslov index zero disks. Notice that $\beta_{\nu_{1}}'$ and $\beta_{\nu_{3}}'$ are parallel. Hence unlike \emph{(Case I.)}, the disks $\beta_{\nu_{1}'} + \beta_{\nu_{2}'} + l \beta_{\nu_{3}'}$ contributes to the potential instead of $\beta_{\nu_{1}}' + l\beta_{\nu_{3}}'$. We emphasize that $\beta_{\nu_{1}'} + \beta_{\nu_{2}'} + l \beta_{\nu_{3}'}$ are \textit{not basic broken disks}. Analogous to \emph{(Case I.)}, let us denote
\[
	\mathcal{A}'_{2} := 
	\left\{\, 
	\beta_{\nu_{1}'} + \beta_{\nu_{2}'} + l \beta_{\nu_{3}}
	\,\, \Big\rvert \,\,
	1 \leq l \leq 
	\# \text{ of blowup centers in } D_{\Sigma, 3} 
	\,\right\}.
\]
The resulting potential is:
\begin{equation}\begin{aligned}\label{eq:case2-after}
W' 	&= \frac{z_{1}}{z_{2}} + T^{-\epsilon}z_{1} + T^{\lambda_{3}}\frac{z_{2}}{z_{1}}
			+ \sum_{l=1}^{k} \left(T^{-\epsilon}z_{1}\right)\left(T^{\lambda_{3}}\frac{z_{2}}{z_{1}}\right)^{l} 
			+ \cdots.
\end{aligned}\end{equation}
Comparing \cref{eq:case2-before2} and \cref{eq:case2-after},
\[
	\support \widetilde{W}'
	=
	\support W'
	\cup
	\left\{\,
	(-k,k+1)
	\,\right\}.
\]

The caveat is that, unlike \emph{(Case I.)}, disks belonging to $\mathcal{A}'_{2}$ are \textit{not basic broken disks}. Nonetheless, the broken disk $\beta_{\nu_{1}'} + \beta_{\nu_{2}}'$ plays the role of a \textit{toric divisor}, in the following sense. Consider the log Calabi-Yau surface $(\underline{X}, \underline{D})$ depicted in red in \cref{fig:case2_toric}, whose (leading order terms of the) potential $\underline{W}$ precisely matches $W'$. The inward normal vectors of $\underline{D}_{\Sigma,1}$, $\underline{D}_{\Sigma,2}$, and $\underline{D}_{\Sigma,3}$ are given by $\underline{\nu}_{1} = (1,-1)$, $\underline{\nu}_{2} = (1,0)$ and $\underline{\nu}_{3} = (-1,1)$, and the energy coefficients are given by $\underline{\lambda}_{1} = 0$, $\underline{\lambda}_{2} = -\epsilon$, and $\underline{\lambda}_{3} = \lambda_{3}$. The non-toric blowup centers on $\underline{D}_{\Sigma, 3}$ can be thought of as being at the same locations and introduce additional mirror charts, or as being near the corner $\underline{D}_{\Sigma, 2} \cap \underline{D}_{\Sigma, 3}$ (see \cref{rmk:critpointsoutside}). It is obvious that  $W'$ and $\underline{W}$ share the same critical points. We need to check whether the \textit{geometric} critical points match. 

\cref{prop:summary} shows that every geometric critical point of $\underline{W}$ comes from the star-shaped Newton polytope of $\underline{W}$, among which non-toric critical points are located near the corner $\underline{D}_{\Sigma,2} \cap \underline{D}_{\Sigma,3}$ (this can be easily seen by mapping out the tropicalization $\mathrm{Trop}(\underline{W})$, or by direct computation as can be found in \cite[Lemma 4.5]{HK23}).
Since the moment polytope of $\underline{X}$ (i.e. the domain of $\underline{W}$) contained within the moment polytope of $X'$ outside an $\epsilon$ neighborhood of $\underline{D}_{\Sigma,1}$, it follows that all geometric critical points of $\underline{W}$ are geometric critical points of $W'$. On the other hand, explicit energy estimation shows that every non-geometric critical point of $\underline{W}$ lies safely away from the moment polytope, i.e. given in $\lambda$-levels (a detailed computation can be found in the proof of \cite[Proposition 4.2]{HK23}). This implies that if a critical point of $\underline{W}$ is non-geometric, then it is also a non-geometric critical point of $W'$. Therefore, every geometric critical point of $W'$ is a geometric critical point of $\underline{W}$  and vice versa, and it follows that $W$ has precisely one less geometric critical point than $\widetilde{W}$.

Note that even though the three leading terms of \cref{eq:case2-after} form a non-convenient cell, there is at least one other term forming another non-convenient $2$-cell sharing the non-convenient edge, due to convexity of the moment polytope. By \ref{assumption}, these $2$-cells are precisely those described in \cref{rmk:nonconvenient}, implying that each $2$-cell gives rises to a non-degenerate critical point which extends to that of the full potential.\\

\emph{(Case III.)} Finally, we consider the case when $D_{\Sigma,2}$ supports two non-toric blowup centers, which is analogous to \emph{(Case II.)}. In this scenario, we have $\nu_{3} = (-1,-1)$. We repeat the process of affine coordinate change in \emph{(Case II.)} twice on the region to the right of the second wall. The resulting region can be seen in \cref{fig:case3_toric}. Analogous to \emph{(Case II.)}, the potential is given as follows:
\begin{equation}\begin{aligned}\label{eq:case3-before}
\widetilde{W}'' = \widetilde{W}''(z_{1}'', z_{2}'',) 
	&= z_{1}'' + z_{2}'' + T^{-\epsilon_{1}}z_{1}''z_{2}'' 
	+ T^{-\epsilon_{1}-\epsilon_{2}} z_{1}''z_{2}''^{2}
	+\frac{T^{\lambda_{3}}}{z_{1}''z_{2}''}
	+ \sum_{l=1}^{k} \left(z_{2}''\right) \left(\frac{T^{\lambda_{3}}}{z_{1}''z_{2}''}\right)^{l}
	+\cdots\\
	&= \frac{z_{1}'}{z_{2}'} + z_{2}' + T^{-\epsilon_{1}}z_{1}' 
	+ T^{-\epsilon_{1}-\epsilon_{2}} z_{1}''z_{2}''
	+ \frac{T^{\lambda_{3}}}{z_{1}'}
	+ \sum_{l=1}^{k} \left(z_{2}\right) \left(\frac{T^{\lambda_{3}}}{z_{1}'}\right)^{l}
	+ \cdots\\
	&= \frac{z_{1}}{z_{2}^{2}} + z_{2} + T^{-\epsilon}\frac{z_{1}}{z_{2}} 
	+ T^{-\epsilon_{1}-\epsilon_{2}} z_{1}
	+ T^{\lambda_{3}}\frac{z_{2}}{z_{1}}
	+ \sum_{l=1}^{k} \left(z_{2}\right) \left(T^{\lambda_{3}}\frac{z_{2}}{z_{1}}\right)^{l}
	+\cdots.
\end{aligned}\end{equation}

\begin{figure}[h!]
\centering
\begin{tikzpicture}[scale=0.07]

		\coordinate (A) at (0,0);
		\coordinate (B) at (0,60);
		\coordinate (a) at (10,0);
		\coordinate (b) at (10,10);
		\coordinate (c) at (20,0);
		\coordinate (d) at (64,0);
		\coordinate (e) at (77,0);
		\coordinate (C) at (80,0);
		\coordinate (D) at (140,60);
		\node[circle, opacity=0] (0) at (-20,64) {};	
		
		\draw[line width = 1pt] (B) -- (0,-80) -- (D);
		\draw[line width = 0.7pt, color=gray!70] (0,10) -- (90,10);				
				
		\draw[dashed, line width = 0.7pt] (10,10) -- (10,-80);
		\draw[dashed, line width = 0.7pt] (20,10) -- (20,-80);		
		\fill[green!30, draw=none, opacity=0.4] (b) -- (10,61) -- (61,61);		
		\fill[green!30, draw=none, opacity=0.4] (20,10) -- (122,61) -- (71,61);		
		
		\draw[line width =0.7pt] (10,10) node[cross] {};
		\draw[dashed, line width = 0.7pt] (10,10)--(10,61);
		\draw[dashed, line width = 0.7pt] (10,10)--(61,61);
		
		\draw[line width =0.7pt] (20,10) node[cross] {};
		\draw[dashed, line width = 0.7pt] (20,10)--(71,61);
		\draw[dashed, line width = 0.7pt] (20,10)--(122,61);

		\draw[line width =0.7pt] (91,11) node[cross, rotate = 45] {};
		\draw [dashed, line width=0.7pt] (91,11)--(68,34);
		\draw [dotted, line width=0.7pt] (68,34)--(44,34);
		\draw [dashed, line width=0.7pt] (44,34)--(34,34);		
		\draw [dotted, line width=0.7pt] (34,34)--(10,34);
		\draw [dashed, line width=0.7pt] (10,34)--(-3,21);

		\draw[line width =0.7pt] (94,14) node[cross, rotate = 45] {};
		\draw [dashed, line width=0.7pt] (94,14)--(72,36);
		\draw [dotted, line width=0.7pt] (72,36)--(46,36);
		\draw [dashed, line width=0.7pt] (46,36)--(36,36);		
		\draw [dotted, line width=0.7pt] (36,36)--(10,36);
		\draw [dashed, line width=0.7pt] (10,36)--(-3,23);

		\draw[line width = 1pt, color= red!80] (112,61)--(38,24)--(27,13)--(27,-53);
		\fill[red!30, draw=none, opacity=0.4] (112,61)--(38,24)--(27,13)--(27,-53)--(141,61);

		\draw[line width = 1pt, color= blue!80] (127,54)--(67,24)--(56,13)--(34, -33)--(34,-39)--(100,-39);
		\draw[line width = 1pt, color= blue!80] (67,24)--(-6,24);
		\draw[line width = 1pt, color= blue!80] (56,13)--(-6,13);
		\draw[line width = 1pt, color= blue!80] (34, -33)--(-5,-72);
		\draw[line width = 1pt, color= blue!80] (34,-39)--(-5,-78);		

			\draw (0,28) node[left] { \small{$\epsilon_{1}$} };
			\draw (0,16) node[left] { \small{$\epsilon_{2}$} };
			\draw (37,10) node[below] { \small{$\epsilon_{1} + \epsilon_{2}$} };

\end{tikzpicture}
\caption{ $(X'', D'' )$ and $(\underline{X}, \underline{D})$. The moment polytope of $(\underline{X}, \underline{D})$ is shown in red. The tropicalization of $W''$ is shown in blue.}%
 \label{fig:case3_toric}
\end{figure}
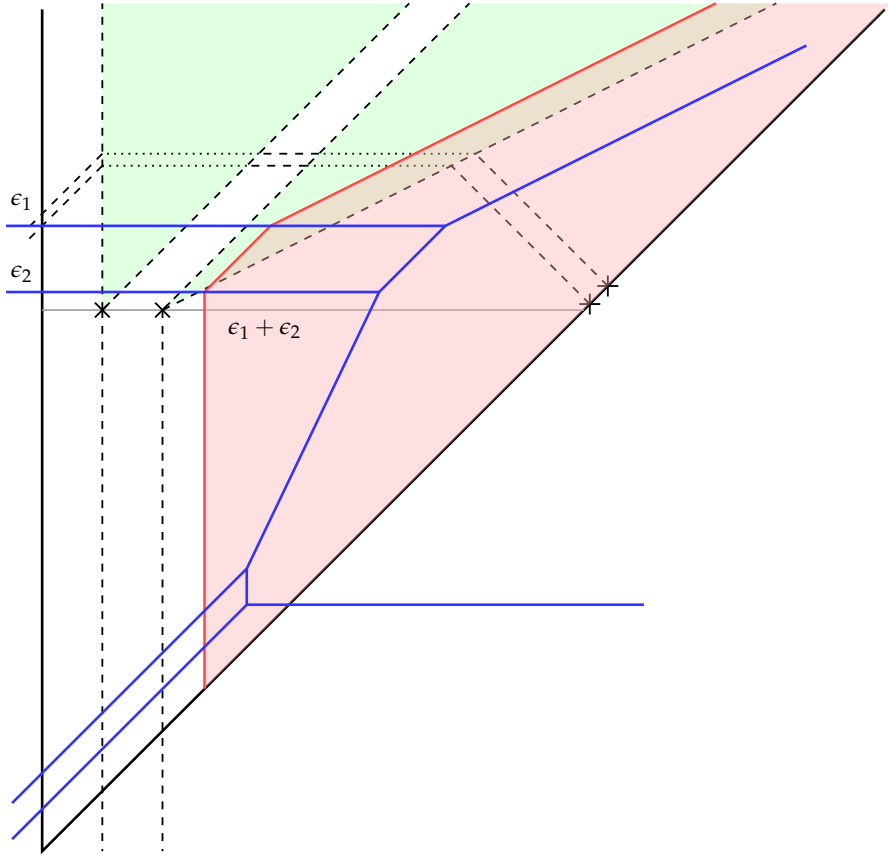

Blowing down $D_{\Sigma, 2}$, the disks $\beta_{\nu_{2}''} + l \beta_{\nu_{3}''} $ becomes Maslov index zero disks, and the disks $\beta_{\nu_{1}''} + 2\beta_{\nu_{2}''} + l \beta_{\nu_{3}''}$ contribute to the potential instead . Explicitly,
\begin{equation}\begin{aligned}\label{eq:case3-after}
W''
	&= \frac{z_{1}}{z_{2}^{2}} + T^{-\epsilon_{1}}\frac{z_{1}}{z_{2}} 
	+ T^{-\epsilon_{1}-\epsilon_{2}} z_{1}
	+ T^{\lambda_{3}}\frac{z_{2}}{z_{1}}
	+ \sum_{l=1}^{k} \left( T^{-\epsilon_{1}-\epsilon_{2}} z_{1}\right) \left(T^{\lambda_{3}}\frac{z_{2}}{z_{1}}\right)^{l}
	+\cdots.
\end{aligned}\end{equation}
We follow the same argument in \emph{(Case II.)} using $(\underline{X}, \underline{D})$ shown in red in \cref{fig:case3_toric}, where $\underline{\nu}_{1} = (-1,2)$, $\underline{\nu}_{2} = (1,-1)$, $\underline{\nu}_{3} = (1,0)$, $\underline{\nu}_{4} = (-1,1)$, and $\underline{\lambda}_{1} = 0$, $\underline{\lambda}_{2} = -\epsilon_{1}$, $\underline{\lambda}_{3} = -\epsilon_{1}-\epsilon_{2}$, $\underline{\lambda}_{4} = \lambda_{4}$. The desired result follows.

The last part of the Proposition is a direct consequence of \cref{prop:summary}; \emph{(CaseI.)} follows immediately, and for \emph{(Case II.)} and \emph{(Case III.)}, one can take use of $\underline{W}$ sharing same geometric critical points with $W$.
\end{proof}

\begin{coro}\label{coro:rank}
Using same notations as in \cref{prop:main},
\[
	\mathrm{dim}\left( \, \mathrm{Jac} \left( W \right) \,\right)
	= \mathrm{rank} \left( \, QH^{\ast} \left(X \right) \, \right)
\]
\end{coro}

\begin{proof}
\begin{equation*}\begin{aligned}
\mathrm{dim}\left( \, \mathrm{Jac} \left( W \right) \,\right) 
&=\mathrm{dim}\left( \, \mathrm{Jac} \left( \widetilde{W} \right)\,\right)  - m 
&& \qquad \text{by \cref{prop:main}  }  \\
&=  \mathrm{rank} \left( \, QH^{\ast} \left(\widetilde{X} \right) \, \right) - m
&& \qquad \text{by Theorem I \cite{HK23}}\\
&= \mathrm{rank} \left( \, QH^{\ast} \left( X \right) \, \right)
\end{aligned}\end{equation*}
\end{proof}



\section{Mirror symmetry for log Calabi-Yau surfaces}\label{sec:4}
We are now ready to prove \ref{thm:intromain}. The only remaining ingredient for the proof is to show that the quantum cohomology $QH^{\ast}\left( X \right)$ is semi-simple. This is essentially due to the fact that the value $W\left( L_{u}, \nabla \right)$, where the Floer cohomology $L_{u}$ twisted by $\nabla$ is non-zero, is an eigenvalue of the quantum multiplication by $c_{1}(X)$ as observed from [Proposition 6.8]\cite{auroux07}, \cite[Theorem 23.12]{fukaya2019spectral} and \cite{Yuan24}.

\begin{prop}[\ref{thm:intromain2}]\label{prop:semisimple}
Let $(X, D)$ be a log Calabi-Yau surface, and suppose that $\omega$ satisfies \ref{assumption}. Then $QH^{\ast}(X)$ is semi-simple.
\end{prop}

\begin{proof}
Suppose that $\mathrm{dim}\left( \, \mathrm{Jac} \left( W \right) \,\right) = \mathrm{rank} \left( \, QH^{\ast} \left( X \right) \, \right) = K$, and let $(\underline{X}, \underline{D})$ and $\underline{W}$ be as constructed in the proof of \cref{prop:main}. We first show that the critical values of $\underline{W}$, and hence that of $W$, are all distinct at each $K$ geometric critical points.

For a non-toric critical point $\alpha$ (i.e. at least one energy minimizing disk of $\alpha$ is from $\mathcal{A}_{j}$), let us consider the following local expression of the leading order terms of $\underline{W}$ :
\[
	z_{2} + z_{1} \prod_{i=1}^{k}{\left(1+T^{-\epsilon_{i}}z_{2}\right)}.
\]
Under the valuation map, each critical point $\alpha_{i}$ is projected onto a distinct vertex of the tropicalization $\mathrm{Trop}(\underline{W})$. The coordinates of each vertex are given by $\mathrm{val}(\alpha_{i}) = \left(\epsilon_{i}+\sum_{j=1}^{i-1}(\epsilon_{j}-\epsilon_{i}), \,\, \epsilon_{i}\right)$, and hence the critical value at each critical point is $\epsilon_{i}$, which are distinct under \ref{assumption}. 

Now suppose that $\alpha$ is a toric critical point. Then by \ref{assumption}, $\alpha$ either lies on the interior of an edge of $\mathrm{Trop}(\underline{W})$, dual to the unique non-convenient edge of the Newton subdivision, or on a trivalent vertex of weight $1$. For the former case, direct computation shows that the critical value at the two critical points (arising from the two unimodal $2$-cells of the Newton subdivision) are distinct. 

If $\alpha$ lies over a trivalent vertex with weight $1$, then $\alpha$ has precisely three associated energy minimizing disks, giving rise to three monomials, say $T^{\lambda_{i}}z^{\partial \beta_{\nu_{i}}}$, $T^{\lambda_{j}}z^{\partial \beta_{\nu_{j}}}$, and $T^{\lambda_{k}}z^{\partial \beta_{\nu_{k}}}$. After appropriate coordinate change, we may assume $\lambda_{j} = 0$ and $\nu_{j} = (0,1)$. Denoting $\mathrm{val}(\alpha) = (x_{0}, y_{0})$, the valuation of the critical value at $\alpha$ is given by
\begin{equation*}\begin{aligned}
\mathrm{val}\left( \underline{W} \big|_{\alpha} \right)
	&= \mathrm{val}\left( T^{\lambda_{i}}\alpha^{\partial \beta_{\nu_{i}}}  + T^{\lambda_{j}}\alpha^{\partial \beta_{\nu_{j}}} + T^{\lambda_{k}}\alpha^{\partial \beta_{\nu_{k}}} \right) \\
	&= \mathrm{min} \left\{ \mathrm{val}\left( T^{\lambda_{i}}\alpha^{\partial \beta_{\nu_{i}}} \right) ,  \mathrm{val}\left( T^{\lambda_{j}}\alpha^{\partial \beta_{\nu_{j}}} \right) , \mathrm{val}\left( T^{\lambda_{k}}\alpha^{\partial \beta_{\nu_{k}}} \right)\right\}\\
	&= \mathrm{val}\left( T^{\lambda_{j}}\alpha^{\partial \beta_{\nu_{j}}} \right) = \lambda_{j} + \langle (x_{0}, y_{0}), \nu_{j}\rangle\\
	&=  y_{0}.
\end{aligned}\end{equation*}
Direct computation shows that $y_{0}$ is given in terms of $\lambda_{i}$ and $\lambda_{k}$, hence again distinct under \ref{assumption}.

On the other hand, recall that if a Lagrangian fiber $L_{u}$ is non-displaceable, then the value $W(L_{u}, \nabla)$ is equal to the eigenvalues of the linear map
\[
	c_{1}(X) \star - : QH^{\ast}\left(X\right) \rightarrow QH^{\ast}\left(X\right)
\]
(e.g. \cite[Proposition 6.8]{auroux07}). It is not difficult to see that the pair $(L_{u}, b)$ associated to a geometric critical point $\alpha$ is indeed non-displaceable; differentiating \cref{eqn:Wfiber} shows that the Floer differential $m_{1}^{b,b}$ vanishes. (Recall that $(L_{u}, b)$ associated to $\alpha$ is obtained via the relation \cref{eqn:relation}, or by considering mirror charts from nodal fibers as explained in \cref{rmk:critpointsoutside}.) Hence it follows that a critical value of $W$ is an eigenvalue of the map $c_{1}(X) \star - $.
That is, $QH^{*}(X)$ has $K$ distinct eigenvalues. Note that $a \star b = 0$ for any $a, b \in QH^{*}(X)$ with eigenvalues $\lambda_{a} \neq \lambda_{b}$. Indeed,
\[
	 \lambda_{a} a \star b
	 = (c_{1} \star a) \star b
	 =(a \star c_{1}) \star b
	 = a \star (c_{1} \star b) 
	 = a \star \lambda_{b} b 
	 = \lambda_{b} a \star b.
\]
Since $QH^{*}(X)$ has rank $K$, it follows that each eigenspace must be dimension $1$.
\end{proof}

Combining \cref{prop:main}, \cref{coro:rank}, \cref{prop:semisimple}, and \cite[Theorem I]{HK23}, we have:
\begin{thm}[\ref{thm:intromain}]\label{thm:main}
Let $(X, D)$ be a log Calabi-Yau surface, and suppose that $\omega$ satisfies \ref{assumption}. Then,
\[
	\mathrm{Jac} \left(W \right) = QH^{\ast} \left( X \right).
\]
\end{thm}


\bibliographystyle{amsalpha}
\bibliography{geometry}

\end{document}